\documentclass[11pt]{amsart}
\usepackage[left=1in, right=1in]{geometry} 
\usepackage{amsmath}
\usepackage[utf8]{inputenc}
\usepackage[all,cmtip]{xy}
\usepackage{amsmath,stmaryrd}
\usepackage{hyperref}
\usepackage{amssymb}
\usepackage{amscd,amsfonts,amsmath,amssymb, bbm,dsfont, comment}
\usepackage{amsmath}
\usepackage{amsfonts}
\usepackage{amssymb}\usepackage{enumerate,epsf,fancyhdr,float,graphicx,tabularx}
\usepackage{latexsym,mathrsfs,multirow}
\usepackage{wasysym}
\usepackage{xypic}
\usepackage[all]{xy}\usepackage[OT2,T1]{fontenc}
\usepackage[modulo,mathlines,displaymath, running]{lineno}
\usepackage{amsmath,mathrsfs,enumerate,wasysym}
\usepackage{xypic,hhline}
\usepackage[all]{xy}
\usepackage[OT2,T1]{fontenc}
\usepackage[modulo,mathlines,displaymath]{lineno}
\usepackage{tikz,tikz-cd}
\usepackage{dashrule}
\usetikzlibrary{matrix}
\usepackage[modulo,mathlines,displaymath]{lineno}

\newtheorem{theorem}{Theorem}[section]

\DeclareSymbolFont{cyrletters}{OT2}{wncyr}{m}{n}\DeclareMathSymbol{\Sha}{\mathalpha}{cyrletters}{"58}
\DeclareMathSymbol{\FSha}{\mathalpha}{cyrletters}{"11}

\renewcommand{\phi}{{\varphi}}

\renewcommand{\geq}{\geqslant}

\newcommand{\links}{\left(\begin{array}{cc}}
\newcommand{\rechts}{\end{array}\right)}
\newcommand{\bai}{\left[\begin{array}{cc}}
\newcommand{\dai}{\end{array}\right]}
\newcommand{\hidari}{\left(\begin{array}{c}}
\newcommand{\migi}{\end{array}\right)}

\newcommand{\C}{\mathbb{C}}

\newcommand{\Q}{\mathbb{Q}}

\newcommand{\Z}{\mathbb{Z}}

\newcommand{\gp}{{\mathfrak p}}
\newcommand{\gq}{{\mathfrak q}}

\newcommand{\ga}{{\mathfrak a}}

\newcommand{\gm}{{\mathfrak m}}

\newcommand{\calA}{\mathcal{A}}

\newcommand{\calC}{\mathcal{C}}

\newcommand{\calO}{\mathcal{O}}

\newcommand{\calV}{\mathcal{V}}
\newcommand{\calW}{\mathcal{W}}

\newcommand{\bA}{\mathbf{A}}

\DeclareMathOperator{\Lie}{Lie}

\newcommand{\Gal}{\operatorname{Gal}}

\newcommand{\Pic}{\operatorname{Pic}}

\newcommand{\Hom}{\operatorname{Hom}}

\newcommand{\GL}{\operatorname{GL}}

\newcommand{\Sel}{\operatorname{Sel}}

\renewcommand{\contentsname}{Contents\\{\footnotesize\normalfont(A table
of contents should normally not be included)}}

\newtheorem{auxiliary proposition}[theorem]{Auxiliary Proposition}

\newtheorem{definition}[theorem]{Definition}

\newtheorem{main conjecture}[theorem]{Main Conjecture}
\newtheorem{main theorem}[theorem]{Main Theorem}
\newtheorem{modesty proposition}[theorem]{Modesty Proposition}

\newtheorem{open problem}[theorem]{Open Problem}

\newtheorem{proposition}[theorem]{Proposition}

\newtheorem{remark}[theorem]{Remark}

\newtheorem{convergence lemma}[theorem]{Convergence Lemma}
\newtheorem{corrected lemma}[theorem]{Corrected Lemma}
\newtheorem{growth lemma}[theorem]{Growth Lemma}
\newtheorem{coefficient lemma}[theorem]{Integrality Lemma}
\newtheorem{interpolation lemma}[theorem]{Interpolation Lemma}
\newtheorem{kernel lemma}[theorem]{Kernel Lemma}
\newtheorem{limit lemma}[theorem]{Limit Lemma}
\newtheorem{tandem lemma}[theorem]{Modesty Lemma}
\newtheorem{zero-finding lemma}[theorem]{Zero-Finding Lemma}

\newcommand{\Ig}{\mathrm{Ig}}
\newcommand{\can}{\mathrm{can}}
\newcommand{\om}{\omega}
\newcommand{\fraka}{\mathfrak{a}}
\newcommand{\bfx}{\mathbf{x}}
\newcommand{\bfn}{\mathbf{n}}

\title[~]{Refined conjectures on Fitting ideals of BDP Selmer groups}

\author{Chan-Ho Kim}
\address[Kim]{
Department of Mathematics and Institute of Pure and Applied Mathematics,
Jeonbuk National University,
567 Baekje-daero, Deokjin-gu, Jeonju, Jeollabuk-do 54896, Republic of Korea
}
\email{chanho.math@gmail.com}
\author[Pal]{Aprameyo Pal}
\address[Pal]{
	 Harish-Chandra Research Institute, Chhatnag Road, Jhunsi, Prayagraj 211 019, India.}
\address[Pal]{
     Homi Bhabha National Institute, Training School Complex, Anushakti Nagar, Mumbai 400 094, India.
}
\email{aprameyopal@hri.res.in}

\author[Ray]{Jishnu Ray}
\address[Ray]{Harish-Chandra Research Institute, Chhatnag Road, Jhunsi, Prayagraj 211 019, India.}
\address[Ray]{Homi Bhabha National Institute, Training School Complex, Anushakti Nagar, Mumbai 400 094, India.
}
\email{jishnuray@hri.res.in; jishnuray1992@gmail.com}
\thanks{Chan-Ho Kim was partially supported 
by the National Research Foundation of Korea(NRF) grant funded by the Korea government(MSIT) (No. RS-2025-16067678, RS-2026-25593558) and
by Global-Learning \& Academic research institution for Master’s$\cdot$Ph.D. Students, and Postdocs (LAMP) Program of the National Research Foundation of Korea (NRF) funded by the Ministry of Education (No. RS-2024-00443714). 
Jishnu Ray gratefully acknowledges support from Inspire Research Grant, Department of Science and Technology, Govt. of India.
}
\keywords{refined Iwasawa theory, BDP $p$-adic $L$-function, anticyclotomic extension}
\subjclass[2020]{Primary: 11R23, Secondary: 11G40, 11F67}
\begin{document}

\begin{abstract}
Under mild assumptions, we prove the strong main conjecture for BDP Selmer groups of elliptic curves with good reduction in the sense of Kurihara's refined Iwasawa theory.
In particular, the structure of the initial Fitting ideal of BDP Selmer groups over finite subextensions in the anticyclotomic $\mathbb{Z}_p$-extension is determined by the values of the corresponding     $p$-adic modular form at CM points.
\end{abstract}
\maketitle
\section{Introduction}
\subsection{Overview}
In mid 1980's, Mazur and Tate initiated to develop the ``tame refinements'' of Birch and Swinnerton-Dyer conjecture \cite{mazur-tate}, \cite{darmon-tate}
\footnote{See also Darmon's article `\href{https://celebratio.org/Gross_BH/article/1000/}{Dick Gross’s marvelous mentoring}' in \emph{Celebratio Mathematica}.}.
This idea leads to Iwasawa theory for elliptic curves (and other motives) over finite abelian extensions of the base field.
Also, in their work, several interesting \emph{refined} conjectures are proposed including the weak vanishing conjecture, the weak main conjecture, and a ``Birch-Swinnerton-Dyer type'' conjecture.
Here, the ``weak'' basically means the inequality, not the equality.

Later, Kurihara also developed this idea further and focused on the structure of Selmer groups rather than the size of Selmer groups \cite{kurihara-fitting, kurihara-invent}. It is now known as \emph{refined Iwasawa theory}.
In particular, he formulated the strong main conjecture, which upgrades Mazur--Tates's refined weak main conjecture to the \emph{equality} between finite layer arithmetic invariants and finite layer analytic invariants, under relevant (but inevitably stronger) assumptions.

Around 2008-10, Emerton--Pollack--Weston proved Mazur--Tate's refined weak main conjecture for modular forms up to a power of $p$ by using the techniques from $p$-adic Hodge theory and the $p$-adic local Langlands correspondence for $\mathrm{GL}_2(\mathbb{Q}_p)$  \cite{epw2}. Later, in \cite{kim-kurihara}, Kurihara and one of us also proved Mazur--Tate's refined weak main conjecture for elliptic curves by utilizing Kobayashi--Pollack's signed Iwasawa theory \cite{pollack-thesis, kobayashi-thesis} and Kurihara's strong main conjecture for elliptic curves by using Kato's zeta elements under certain hypotheses.

Recently, the idea of this finite layer Iwasawa theory extends to several other settings including Rankin--Selberg products of modular forms \cite{cauchi-lei} and $\mathbb{Z}^2_p$-extensions \cite{dion}.
The main goal of this article is to apply this idea to the so-called BDP setting \cite{bertolini-darmon-prasanna-duke} and illustrate some interesting consequences.
In particular, we obtain an analogue of Kurihara's strong main conjecture for BDP Selmer groups for both good ordinary and supersingular reduction under certain assumptions.
One remarkable feature  is that the Fitting ideal of BDP Selmer groups over finite subextensions of the anticyclotomic $\mathbb{Z}_p$-extension is principal for both cases. It is unexpected, as it is not true in the classical setting.

Refined Iwasawa theory at finite layers can be more subtle than the standard Iwasawa theory at the infinite layer as we cannot directly apply the theory of Iwasawa modules to finite layer objects and  we cannot ignore ``finite errors''. Also, the structure of modules over group rings at finite layers is much more complicated than the structure of modules over the Iwasawa algebra.
Delicate care is necessary to control finite error terms while descending from the infinite $\Z_p$-extension to finite layer group rings.

\subsection{Review of some recent results on refined Iwasawa theory}
Although Mazur--Tate's original formulation covers on finite abelian extensions of the base field, most works focus on finite subextensions of a $\mathbb{Z}_p$-extension of the base field since their methods are essentially rooted in Iwasawa theory. 
See \cite{kataoka-thesis, kataoka-stark-equivariant} for equivariant refinements.

\subsubsection{The cyclotomic (classical) setting}\label{sec:cyc}
Let $p$ be an odd prime. Let $E$ be an elliptic curve over $\Q$  and $\Q_\infty$ be the cyclotomic $\Z_p$-extension of $\Q$ with subextensions $\Q_n$ such that $\Gal(\Q_n/\Q)\cong \Z/p^n\Z$.  Let $\Lambda=\Z_p \llbracket\Gal(\Q_\infty/\Q) \rrbracket$ and $\Lambda_n=\Z_p[\Gal(\Q_n/\Q)]$.

Suppose that $E$ has good ordinary reduction at $p$ with no rational $p$-torsion, $a_p(E) \not\equiv 1 \pmod{p}$, and $p$ does not divide the Tamagawa number of $E$. Let $f \in S_2(\Gamma_0(N))$ be the newform attached to $E$ and $f_\alpha$ be the $p$-stabilized form of $f$ with unit root $\alpha$.
If the Iwasawa main conjecture  
$$\mathrm{char}_{\Lambda}(\Sel(\Q_\infty, E[p^\infty])^\vee)=(L_p(\Q_\infty,f_\alpha))$$
is true, then it is not difficult to show that 
$$\mathrm{Fitt}_{\Lambda_n}(\Sel(\Q_n,E[p^\infty])^\vee)=(\theta_n(f_\alpha)) = (\theta_n(f), \nu_{n-1, n} (\theta_{n-1}(f)) ).$$
See \cite[\S 2.2]{kim-kurihara} for example.
Here, $\theta_n(f_\alpha)$ is the projection of $L_p(\Q_\infty,f_\alpha)$ to $\Lambda_n$ and 
$\theta_n(f) \in \Lambda_n$ is the \emph{Mazur--Tate element} of $f$ over $\mathbb{Q}_n$ defined using modular symbols as in \cite[\S 1.2.2]{kim-kurihara}.
Here $\nu_{n-1,n}$ is the trace map $\Lambda_{n-1}\rightarrow \Lambda_n$ defined by $\sigma \mapsto \sum_{\tau \mapsto \sigma}\tau$ for $\sigma \in \Gal(\Q_{n-1}/\Q)$ and $\tau$ runs over all elements of $\Gal(\Q_{n}/\Q)$ projecting to $\sigma$. 
Also, $\mathrm{Fitt}_{\Lambda_n}(-)$ denotes the initial Fitting ideal over $\Lambda_n$.
In this case, the Fitting ideal of $\Sel(\Q_{n},E[p^\infty])^\vee$ is principal.
However, it is expected that the corresponding Fitting ideal is not principal when $E$ has supersingular reduction at $p$. 

Suppose now that $E$ has supersingular reduction at $p$ and we keep our running assumptions that $E(\Q)[p]$ is trivial and $p$ does not divide the Tamagawa number of $E$.
When $p \nmid \frac{L(E,1)}{\Omega_E}$ and the residual representation $\overline{\rho}_E$ is surjective, Kurihara confirmed his strong main conjecture \cite[Theorem 0.1(4)]{kurihara-invent}
\begin{equation}\label{nord_cyclo}
\mathrm{Fitt}_{\Lambda_n}(\Sel(\Q_{n},E[p^\infty])^\vee)=\big(\theta_n(f), \nu_{n-1,n}(\theta_{n-1}(f)\big). 
\end{equation}
Note that his work did not assume $a_p(E)=0$. 
Under another set of hypotheses including $a_p(E)=0$, one of the authors and Kurihara proved \eqref{nord_cyclo} by appealing to  Kobayashi's $\pm$-main conjecture (see \cite[Conjecture 3.5]{kim-kurihara}).
In this situation, an additional hypothesis on \textit{fine} Selmer group (see \cite[Theorem 1.20]{kim-kurihara}) is imposed in order to guarantee that the ``error terms'' vanish while descending from the infinite level to finite levels.

\subsubsection{The anticyclotomic direction}
Suppose now that $K$ is an imaginary quadratic field of odd discriminant $-D_K<-4$ with $(D_K, Np)=1$ 
where $N$ is the conductor of $E$. Let $K_\infty$ be the anticyclotomic $\Z_p$-extension of $K$ with subextensions $K_n$ of degree $p^n$ over $K$. 
We write $N = N^+ \cdot N^-$ where a prime divisor of $N^-$ is inert in $K$ and a prime divisor of $N^+$ splits in $K$. 
Denote by $N(\overline{\rho}_E)$ the Artin conductor of the residual representation $\overline{\rho}_E$.
We assume that $p$ splits in $K$ and
 $N^-$ is a square free product of an \textit{odd} number of primes (also known as the \textit{definite} setting).

When $E$ has good ordinary reduction at $p \geq 5$, we further assume that
\begin{itemize}
    \item $\overline{\rho}_E$ is surjective,
    \item $N = N(\overline{\rho}_E) $ so that $p$ does not divide the Tamagawa factors of $E$, and
    \item $a_p(E) \not\equiv 1 \pmod{p}$.
\end{itemize}
In this situation, one of us showed the strong version of the refined main conjecture (cf. \cite[Thm. 1.1 and Lem. 2.2]{kim2026})
$$\mathrm{Fitt}_{\Lambda_n}(\Sel(K_n,E[p^\infty])^\vee)=(\theta_n(f_\alpha))^2.$$
Here $\theta_n(f_\alpha)$ is the \emph{Bertolini--Darmon's theta element} attached to the $p$-stabilized  form $f_\alpha$ with unit root $\alpha$ (cf. \cite[\S 2.2, \S 2.3]{kim2026}). 

When $E$ has supersingular reduction at $p \geq 5$ (so $a_p(E)=0$), we further assume that
\begin{itemize}
    \item $\overline{\rho}_E$ is surjective,
    \item $K_\infty$ is totally ramified at every prime lying above $p$, and
    \item if a prime $\ell$ divides $N^-$ and $\ell^2 \equiv 1 \pmod{p}$, then $\overline{\rho}_E$ is ramified at $\ell$.  
\end{itemize}
Then the \textit{weak} version of the refined main conjecture could be obtained, i.e. 
$$\left( \theta_n(E/K) \right)^2  = \left( \theta_n(E/K) \cdot \iota(\theta_n(E/K)) \right) \subseteq \mathrm{Fitt}_{\Lambda_n}(\Sel(K_n,E[p^\infty])^\vee).$$
Here, $\iota$ is the involution on $\Lambda_n$ defined by inverting group-like elements.
Note that this uses the one-sided divisibility of the signed main conjectures.
; these conjectures are necessary to assume in order to  upgrade this weak refined result to a stronger equality of ideals (see \cite[Rem. 3.6]{kim2026anticyclotomic}). 
More recently, Shii has proved this weak refined conjecture in the definite setting when $p$ is inert in $K$ (cf. \cite{shii2025}).

\subsection{Our work in this paper}
In this paper we will work over the anticyclotomic $\mathbb{Z}_p$-extension of an imaginary quadratic field for the BDP-Selmer group and the BDP $p$-adic $L$-function and we cover both the ordinary and the supersingular cases. In the supersingular case we will work in the \textit{indefinite} setting (i.e. $N^-$ is the square free product of an even number of primes). This indefinite case has not been dealt in the existing literature yet and is our main interest in this paper.  We will first recall  the BDP $p$-adic $L$ function and the corresponding BDP Selmer group  the following section. 
\subsection{The BDP $p$-adic $L$-function following  Castella--Hsieh}
\label{subsec:setup}
Suppose $f$ is a newform of weight $2r \geq 2$ and of level $\Gamma_0(N)$. Suppose $K/\Q$ is an imaginary quadratic field of odd discriminant $-D_K<-3$ coprime to $Np$ and assume that $\chi: G_K:=\Gal(\overline{\Q}/K) \rightarrow F^\times$ is a locally algebraic anticyclotomic character. Let $ p$ be an odd prime not dividing $N$ which splits as $p \mathcal{O}_K= \gp \cdot \gp^c$.  Let $V_f$ the self dual Tate twist of the $p$-adic Galois representation associated to $f$ by Deligne and consider the $G_K$-representation $V_{f,\chi}:=V_f(r) \otimes \chi$. This is a conjugate self-dual representation and the associated Rankin $L$-series $L(f, 
\chi, s)$ satisfies a functional equation relating $s$ and $2r-s$. Now suppose $\Gamma:=\Gal(K_\infty/K)$ be the Galois group of the anticyclotomic $\Z_p$-extension of $K$ and $\Lambda$ be the Iwasawa algebra $\Z_p \llbracket \Gamma \rrbracket$. 
Let $\calW$ be the completion of the ring of integers of the maximal unramified extension of $\mathbb{Q}_p$.

Write $c=c_0p^s$ with $p \nmid c_0$. Let us suppose that  $\chi=\psi \phi_0$ where $\psi$ is an anticyclotomic character of infinity type $(r,-r)$ and conductor $c_0\mathcal{O}_K$ and $\phi_0$ is a $p$-adic character of $\Gamma$. Castella--Hsieh define the $p$-adic $L$-function $$\mathscr{L}_{\gp, \psi}(f)\in \Lambda^{\mathrm{ur}} = \calW  \llbracket \Gamma  \rrbracket$$
interpolating the central critical values $L(f,\psi \phi, r)$ as $\phi$ runs over a Zariski-dense subset of $p$-adic characters of $\Gamma$.
This $p$-adic $L$-function was first defined in the work of Bertolini--Darmon--Prasanna \cite{bertolini-darmon-prasanna-duke} where they deduced a formula relating the values of
$\mathscr{L}_{\gp, \psi}(f)$ at \textit{unramified characters outside} the range of interpolation to the $p$-adic Abel-Jacobi images of generalized Heegner cycles. For the rest of the introduction, we assume that $f$ is associated with an elliptic curve $E$ over $\Q$ and hence $r=1$. Consequently, $\psi$ is an anticyclotomic Hecke character of infinity type $(1,-1)$. As in \cite[p. 540, line -2]{CGLS2022}, for every character $\xi$ of $\Gamma$, define $\mathcal{L}_E\in \Lambda^\mathrm{ur}$ as $$\mathcal{L}_E(\xi)=\mathscr{L}_{\gp,\psi}(\psi^{-1}\xi)^2.$$ 

The Iwasawa main conjecture predicts that $\mathcal{L}_E(\xi)$ generates the characteristic ideal of the Pontryagin dual of a certain Selmer group attached to $f$ over $K_\infty$ which is the so-called BDP Selmer group $\Sel^{\emptyset, 0}(K_\infty, E[p^\infty])$ (cf. \S \ref{sec:BDP Selmer}); its Pontryagin dual is denoted as $X^{\emptyset,0}(f)$.  We can also define the BDP Selmer group $\Sel^{\emptyset, 0}(K_n, E[p^\infty])$ over the finite layers $K_n$ (cf. \S \ref{sec:BDP Selmer}) and we have $X^{\emptyset, 0}=\varprojlim_n \Sel^{\emptyset, 0}(K_n,E[p^\infty])^\vee$.  The \textit{integral} version of the Iwasawa main conjecture predicts (under various hypotheses depending, among others,  on the reduction type of $E$ at $p$, cf. \S \ref{sec:ord} and \S \ref{sec:supersingular}) that  
\begin{equation}\label{BDP}
   X^{\emptyset,0}\text{ is a torsion } \Lambda\text{-module and }  \mathrm{Char}_{\Lambda}(X^{\emptyset,0})\Lambda^{\mathrm{ur}}=(\mathcal{L}_E) 
\end{equation}
as ideals in $\Lambda^{\mathrm{ur}}.$
This conjecture is known to be true in the ordinary case by \cite{burungale-castella-skinner-gl2} and in the non-ordinary case with $a_p=0$ by 
\cite{castella-wan-perrin-riou-ss} (see \S \ref{sec:BDPselmer} for detailed necessary hypotheses).

We now list all our necessary hypotheses in order to state our main results precisely.

For each prime number $v$ of $K$, let $c_v(E):=\big|E(K_v)/E_0(K_v)\big|$ be the Tamagawa number of $E$ at $v$. The  assumption we impose is
\begin{equation}
\tag{Tam} \text{$p\nmid c_v(E)$ for all primes $v\,|\,N$.}
\end{equation}
We also assume the following condition:
\begin{equation}
\tag{$\mathrm{H^0}$} \text{$H^0(K_v,A)=0$ for all primes $v\,|\,p$.}
\end{equation}
Furthermore we will assume all the hypotheses necessary to have the main conjecture for BDP Selmer groups \eqref{BDP}.
When $E$ is of good ordinary reduction at $p$, these hypotheses are listed in  \S \ref{sec:ord} and when $E$ has supersingular reduction at $p$ with $a_p(E)=0$ see \S \ref{sec:supersingular} for the list of hypotheses.
\begin{theorem}
    Assume the hypotheses $\mathrm{(Tam)}$, $\mathrm{(H^0)}$  with the discriminant $-D_K<-3$ such that $(D_K,pN)=1$.
    \begin{enumerate}
        \item If $E/\Q$ has good ordinary reduction at $p$, assume the hypotheses (1) - (5) in \S \ref{sec:ord}.
        \item  If $E/\Q$ has supersingular reduction at $p$, assume the hypotheses (1) - (6) in \S \ref{sec:supersingular}.
    \end{enumerate}
Then
$\mathrm{Fitt}_{\Lambda_n}\left(\Sel^{\emptyset, 0}\big(K_n, E[p^\infty])\big)^\vee\right)$  is the principal ideal generated by the square of a BDP analogue of theta element $\theta_{n}(f)$ which is explicitly  defined as a collection of the values of the corresponding $p$-adic modular forms at CM points in \S \ref{def_theta} in $\Lambda_n$.
\end{theorem}
One of the interesting features is that the shape of the formula of the BDP theta element $\theta_n(f)$ is independent of the reduction type of $E$ at places above $p$.
In particular, as mentioned before, this initial Fitting ideal of the BDP Selmer group at $K_n$ is principal even when $E$ has supersingular reduction at $p$.
In the cyclotomic formulation, it never happens as explained in \S \ref{sec:cyc}.


\subsubsection*{Organization of the article}
In \S \ref{sec:2.1pre} - \S \ref{sec:2.4BDp} we recall the definition of the BDP $p$-adic $L$-function following \cite{castella-hsieh-heegner-cycles}. The descent argument of the reduction of this BDP $p$-adic $L$-function at finite layer is treated in \S \ref{sec:2.5BDP} with the BDP theta element defined in \S \ref{def_theta}. Next we discuss the Iwasawa main conjecture for BDP Selmer group in \S \ref{sec:BDP Selmer}. After recalling some algebraic preliminaries in \S \ref{quick}, we finally show in \S \ref{final} that $\theta_{n}(f)$ generates the initial Fitting ideal of the Pontryagin dual of the BDP Selmer group over $K_n$.

\section{BDP $p$-adic $L$-functions and BDP theta elements}\label{sec:preliminaries}
In this section, we review the analytic construction of the BDP $p$-adic $L$-function following the exposition given in \cite{castella-hsieh-heegner-cycles} closely.
\subsection{Preliminaries on modular forms} \label{sec:2.1pre}
\subsubsection{Igusa towers} Assume that the prime $p$ is coprime to  $N\geq 3$, and let $\Ig(N)_{/\Z_{(p)}}$ be the Igusa scheme over $\Z_{(p)}$.  This is the moduli space parameterizing elliptic curves with $\Gamma_1(Np^\infty)$-level structure. For each locally noetherian scheme $S$ over $\Z_{(p)}$, $\Ig(N)(S)$ is the set of isomorphism classes of pairs $(A,\eta)$ consisting of an elliptic curve $A$ over $S$ and a $\Gamma_1(Np^\infty)$-level structure 
\begin{equation}\label{eq:immersion}
  \eta=(\eta^{(p)},\eta_p):\mu_N\oplus \mu_{p^\infty}\hookrightarrow A[N]\oplus A[p^\infty]  
\end{equation}
Here $\mu_N$ is the group scheme of $N$-th  roots of unity and \eqref{eq:immersion} is
 an immersion as group schemes over $S$. Letting $\mathbb{H}$ be the complex upper half-plane. We have an explicit map 
\[ \mathbb{H}\times \mathrm{GL}_2(\hat{\Q})\to \Ig(N)(\C),\quad x=(\tau_x,g_x)\mapsto [(A_x,\eta_x)]. \]
For explicit construction of $[(A_x,\eta_x)]\in \Ig(N)(\C)$ corresponding to each $x$, the reader is advised to consult \cite[p. 573]{castella-hsieh-heegner-cycles}.

\subsubsection{Geometric modular forms}
Let $({\rm Tate}(q),\eta_{\mathrm{can}},\omega_{\can})$ be the Tate elliptic curve $\mathbb{G}_m/q^\Z$
with the canonical level structure $\eta_{\can}$ and the canonical differential $\om_{\can}$ over $\Z(\!(q)\!).$
\begin{definition}
Let  $B$ be a $\Z_{(p)}$-algebra, $k$ be an integer and $C$ be a $B$-algebra. Let $[(A,\eta)]$ be a point in $ \Ig(N)(C)$ and $\omega$ be a basis  of $H^0(A,\omega_{A/C})$ over $C$.  A geometric modular form $f$ of weight $k$ on $\Ig(N)$ defined over $B$
is a rule that assigns to every triple $(A,\eta,\omega)$ over $C$,  a value $f(A,\eta,\omega)\in C$ such that the following conditions are met:
\begin{itemize}
\item[(G1)] $f(A,\eta,\omega)=f(A',\eta',\omega')\in C$ if $(A,\eta,\omega)\cong (A',\eta',\omega')$ over $C$.
\item[(G2)] For any  $B$-algebra homomorphism $\varphi:C\to C'$, we have
$f((A,\eta,\omega)\otimes_C C')=\varphi(f(A,\eta,\omega)).$
\item[(G3)] $f(A,\eta,t\omega)=t^{-k}f(A,\eta,\omega)$ for all $t\in C^\times$.
\item[(G4)] The value of $f({\rm Tate}(q),\eta_{\can},\om_{\can})$ is an element of $ B  \llbracket q  \rrbracket$. This power series in 
$ B  \llbracket q  \rrbracket$
is called the algebraic Fourier expansion of $f$. 
\end{itemize}
\end{definition}
Suppose $f$ is a geometric modular form of weight $k$  over a subring $\calO\subset\C$.  Define a holomorphic function $\mathbf{f}:\mathbb{H}\times\GL_2(\hat{\Q})\to\C$ by the rule
\[\mathbf{f}(x)=f(A_x,\eta_x,2\pi i dw),\quad  x\in \mathbb{H}\times \GL_2(\hat{\Q}),\]
where $w$ is the standard complex coordinate of $A_x=\C/L_x$. Here $L_x$ is the period lattice of $A_x$ attached to the standard differential form $dw$ (cf. \cite[p. 573]{castella-hsieh-heegner-cycles}).

Let $U_0(Np^n)=\{g\in \GL_2(\hat{\Z})\mid g\equiv \begin{pmatrix}
  * & *\\ 
  0 & *
\end{pmatrix}\pmod{Np^n}\}$.
We say that the geometric modular form $f$ is of level $\Gamma_0(Np^n)$ if $\mathbf{f}(\tau,gu)=\mathbf{f}(\tau,g)$ for all $u\in U_0(Np^n)$. 
One can  show that the function $\mathbf{f}(-,1):\mathbb{H}\to\C$ is a classical elliptic modular form of weight $k$ with analytic Fourier expansion
\[\mathbf{f}(\tau,1)=\sum_{n\geq 0}\mathbf{a}_n(f) e^{2\pi i n\tau}.\]
Furthermore, one can show that 
$f({\rm Tate}(q),\eta_{\can},\om_{\can})=\sum_{n\geq 0}\mathbf{a}_n(f) q^n\in \calO  \llbracket q  \rrbracket.$

\subsubsection{$p$-adic modular forms}
Let $R$ be a $p$-adic ring, and define the formal completion of $\Ig(N)$ over $R$ as $\widehat{\Ig}(N)_{/R}:=\varinjlim_m \Ig(N)_{/R/p^m R}$. Let $V_p(N,R)$ be the space of $p$-adic modular forms of level $N$ which is defined as 
\begin{align*}
V_p(N,R)&:=H^0(\widehat{\Ig}(N)_{/R},\calO_{\widehat{\Ig}(N)_{/R}})\\
&=\varprojlim_m H^0(\Ig(N),\calO_{\Ig(N)}\otimes R/p^mR).
\end{align*}
An element $f \in V_p(N,R)$ (i.e. a $p$-adic module form) is said to be of weight $k\in\Z_p$ if for every $u\in\Z_p^\times$, we have
$$f(A,\eta)=u^{-k}f(A,\eta^{(p)},\eta_p u),\quad [(A,\eta)]=[(A,\eta^{(p)},\eta_p)]\in\widehat{\Ig}(N)_{/R}.$$

Suppose $f$ is a geometric modular form defined over $R$ and $C$ be a complete local $R$-algebra. Then, the associated $p$-adic modular form $\hat{f}$ is defined by 
\[\widehat{f} (A,\eta)=f(A,\eta,\widehat{\om}(\eta_p)),\quad[(A,\eta)]\in\widehat{\Ig}(N)_{/R}.
\]
Here $\widehat{\omega}(\eta_p)\in\Lie(A)=\Lie(\widehat{A})\cong C$ is a differential which originates from the 
 isomorphism $\widehat{\eta}_{p}:\widehat{\mathbf{G}}_m\cong \widehat{A}$ ($\widehat{A}$ is the formal group of $A$) obtained using the $p^\infty$-level structure $\eta_p$.

We note that if $f$ is a geometric modular form of weight $k$ and level $\Gamma_0(Np^n)$,
then $\widehat{f}$ is a $p$-adic modular form of weight $k$.

\subsection{CM points}
Suppose $K$ is an imaginary quadratic field of discriminant $-D_K<0$ and $p>2$ is a prime that splits in $\calO_K$ as $p\calO_K=\gp\bar{\gp}$ where $\gp$ is the prime ideal above $p$ determined by the embedding $\iota_p:\bar{\Q} \rightarrow \C_p$. 
Assume that $p \nmid h_K$ where $h_K$ is the class number of $K$.
Let
$$\vartheta = \dfrac{D' + \sqrt{-D_K}}{2} \in K$$
where $D' = D_K$ if $2 \nmid D_K$ and $D' = D_K/2$ if $2 \mid D_K$. Let $z \mapsto \bar{z}$ be the complex conjugation on $\C$.
Assume that $N\calO_K=\mathfrak{N}\bar{\mathfrak{N}}$ for some ideal $\mathfrak{N}$ of $\calO_K$.
Let $c$ be a positive integer, let $\calO_c:=\Z+c\calO_K$ be the order of conductor $c$, and let $K_c$ be the ring class field of $K$ of conductor $c$. Let $\mathfrak{a}$ be a fractional ideal of $\calO_c$, and let $a\in \widehat{K}^\times$ with $a\widehat{K}\cap \calO_c=\mathfrak{a}$. To such a pair $(\mathfrak{a},a)$, one can associate a $\C$-pair $(A_\mathfrak{a},\eta_a)$ of complex CM elliptic curves with $\Gamma_1(Np^\infty)$-level structure (see \cite[p. 576]{castella-hsieh-heegner-cycles} for the construction). 

Let $\calV$ be the valuation ring $\iota_p^{-1}(\calO_{\C_p})\cap K^\mathrm{ab}$. Using the theory of complex multiplication 
along with the criterion of Serre--Tate  it can be shown that 
$[(A_\mathfrak{a},\eta_a)]$ is actually defined over $\calV_0$ (a discrete valuation ring in $\calV$) and hence it belongs to $\Ig(N)(\calV_0)$. Henceforth, the point
 $[(A_\fraka,\eta_a)]\in \Ig(N)(\calV)$ is called the CM point attached to $(\fraka,a)$.

If $\fraka$ is an integral ideal of $\calO_c$ which is coprime to $\mathfrak{N}\mathfrak{p}$, then we write $(A_\fraka,\eta_\fraka)$ for the pair $(A_\fraka,\eta_a)$ with $\mathfrak{q}$-component $a_\gq=1$ for all prime $\gq\mid\mathfrak{N}\gp$. If $\fraka=\calO_c$, then we write $(A_c,\eta_c)$ for $(A_{\calO_c},\eta_{\calO_c})$. 

Let $Y_1(Np^n)/\Q$ be the usual open modular curve with $\Gamma_1(Np^n)$ level structure. For  $g\in\GL_2(\hat{\Q})$, let $[(\vartheta,g)]$ be the image of $(\vartheta,g)$ in $\varprojlim_n Y_1(Np^n)(\C)=\Ig(N)(\C)$.
By Shimura's reciprocity law for CM points  one obatins that $[(\vartheta,g)]\in \Ig(N)(K^{\mathrm{ab}})$ and
$$\mathrm{rec}_K(a)[(\vartheta,g)]=[(\vartheta,\bar{a}g)]\quad(a\in\widehat{K}^\times).$$
Here $\mathrm{rec}_K:K^\times\backslash \widehat{K}^\times\mapsto \Gal(K^{\mathrm{ab} }/K)$ is the \emph{geometrically normalized} reciprocity law map.
More generally, if $a\in\widehat{K}^{(cp)\times}$ and $\fraka=a\widehat{\calO}_c\cap K$ is a fractional ideal of $\calO_c$, then define 
\[\sigma_\fraka:=\mathrm{rec}_K(a^{-1})|_{K_c(\gp^\infty)}\in \Gal(K_c(\gp^\infty)/K).\]
Here $K_c(\gp^\infty)$ is the compositum of  the ray class field of $K$ of conductor $\gp^\infty$ and $K_c$.
Under this setting we have (see \cite[eq. (2.5)]{castella-hsieh-heegner-cycles} for clarification)
$$x_\fraka:=[(A_\fraka,\eta_a)]=[(\vartheta, \bar{a}^{-1}\xi_c)]=x_c^{\sigma_\fraka}\in \Ig(N)(K_c(\gp^\infty)).$$
Here $\xi_c $ is a certain element of $\GL_2(\hat{\Q})$ defined as in \cite[before eq. (2.3)]{castella-hsieh-heegner-cycles}.

Suppose $\fraka$ is a fractional ideal of $\calO_c$ coprime to $\gp\mathfrak{N}$ and suppose that $p\nmid c$,
then $(A_\fraka,\eta_\fraka)$ is defined over $\calV^{\mathrm{ ur}}:=\calW\cap K^{\mathrm{ ab}}.$
\subsection{Serre--Tate coordinates}
Suppose $\mathbf{x}=[(A_0,\eta)]\in  \Ig(N)(\bar{\mathbb{F}}_p)$ and let $\widehat{ S}_{\mathbf{x}}\hookrightarrow \Ig(N)_{/\calW}$ be the local deformation space of $\bfx$ over $\calW$. Suppose $A_0^t$ is the dual abelian variety of $A_0$ and $T_p(A_0^{t})=\varprojlim_n A_0^{t}[p^n](\bar{\mathbb{F}}_p)$ is the $p$-adic Tate module of $A_0^{t}$. Let $\lambda_{\mathrm{can}}: A_0\cong A_0^{ t}$ be the canonical principal polarization.

The $p^\infty$-level structure $\eta_p$ gives a point $P_\bfx\in T_p(A_0^{t})$.
Let
$q_\calA: T_p(A_0)\times T_p(A_0^{ t})\mapsto 1+\gm_R$ be the Serre--Tate bilinear form attached to a deformation $\calA_{/R}$ over a local Artinian ring $(R,\gm_R)$.
The \emph{canonical Serre--Tate coordinate} $t:\widehat{S}_{\mathbf{x}}\mapsto\widehat{\mathbf{G}}_m$ is defined by \[t(\calA):=q_\calA(\lambda_{\mathrm{can}}^{-1}(P_\bfx),P_\bfx).\]
This gives an identification $\calO_{\widehat{S}_\bfx}=\calW \llbracket t-1 \rrbracket.$

Let  $f\in V_p(N,\calW)$.
We define the  $t$-expansion $f(t)$ of $f$ around $\bfx$ as \[f(t)=f|_{\widehat{ S}_\bfx}\in \calW \llbracket t-1 \rrbracket.\]
Let $\mathrm{d} f$ be the $p$-adic measure on $\Z_p$ with the property
\[\int_{\Z_p}t^x \mathrm{d} f(x)=f(t).\]
Furthermore, for any continuous function $\phi:\Z_p\mapsto\calO_{\C_p}$, we define $f\otimes \phi(t)\in\calO_{\C_p} \llbracket t-1 \rrbracket$ by
\[f\otimes\phi(t)=\int_{\Z_p}\phi(x)t^x \mathrm{d} f=\sum_{n\geq 0}\int_{\Z_p}\phi(x){x\choose n} \mathrm{d} f(x)\cdot (t-1)^n.\]
Assume that $c$ is a positive integer coprime to $p$. Define $\mathrm{N}(\fraka)$ by \begin{align*}\mathrm{N}(\fraka)&:=\text{degree of the $\Q$-isogeny }\C/\calO_K\mapsto\C/\fraka^{-1}
\\
&=c^{-1}\#(\calO_c/\fraka)=c^{-1}|a|_{\mathbf{A}_K}^{-1}.\end{align*}

Let $\fraka$ be an integral ideal of $\calO_c$ coprime to $c\mathfrak{N} p$, and as before let $a\in\widehat{K}^{(c\mathfrak{N} p)\times}$ be such that $\fraka=a\widehat{\calO}_c\cap K$. Attached to the pair $(\fraka, a)$, we have a CM point
 $x_\fraka=[(A_\fraka,\eta_\fraka)]\in\Ig(N)(\calV)$. 
Define $\bfx_\fraka:=x_\fraka\otimes_\calV\bar{\mathbb{F}}_p$ and 
 let $t$ be the canonical Serre--Tate coordinate of $\bfx_\fraka$.
For $z\in\Q_p$, define
$$\bfn(z):=\left( \begin{array}{cc}
1 & z \\
0 & 1
\end{array} \right)\in\GL_2(\Q_p)\subset\GL_2(\hat{\Q}).$$
Define $x_\fraka*\bfn(z)$ as 
\[x_\fraka*\bfn(z):=[(\vartheta,\bar{a}^{-1}\xi_c\bfn(z))]\in\Ig(N)(\calV).\]
Let $q$ be any prime, $\zeta_{q^n}:=\mathrm{exp}(\frac{2 \pi i}{q^n})$ and $\phi:\Z_q^\times \rightarrow \C^\times$ be a continuous character of conductor $q$. Define the Gauss sum $\mathfrak{g}(\phi)$ as $$\mathfrak{g}(\phi)=\sum_{u \in (\Z/q^n\Z)^\times}\phi(u)\zeta_{q^n}^u.$$
We quote the following proposition from \cite[Proposition 3.3]{castella-hsieh-heegner-cycles}
\begin{proposition}\label{prop:tensor}
    Let $f\in V_p(N,\calW)$  with $t$-expansion $f(t)$ around $x_\fraka\otimes\bar{\mathbb{F}}_p$. Set
\[f_\fraka(t):=f(t^{\mathrm{N}(\fraka)^{-1}\sqrt{-D_K}^{-1}}).\]
Suppose $\phi:(\Z/p^n\Z)^\times \mapsto \calO_{\C_p}^\times$ is a primitive Dirichlet character, then
\[f_\fraka\otimes \phi(x_\fraka)=p^{-n}\mathfrak{g}(\phi)\sum_{u\in(\Z/p^n\Z)^\times}\phi^{-1}(u)\cdot f(x_\fraka*\mathbf{n}(up^{-n})).\]

\end{proposition}

\begin{definition}
Let 
$$\chi : K^\times \backslash \mathbb{A}^\times_K \to \mathbb{C}^\times$$
be a Hecke character.
\begin{enumerate}
\item We say that\textit{ $\chi$ has infinity type $(m,n)$}
if $$\chi_{\infty} (z) = z^m \overline{z}^n .$$
\item We say that \textit{$\chi$ is anticyclotomic} if $\chi(\mathbb{A}^\times_{\mathbb{Q}}) = 1$.
\end{enumerate}
\end{definition}
Let $\chi_{\mathfrak{q}} : K^\times_{\mathfrak{q}} \to \mathbb{C}^\times$ be the $\mathfrak{q}$-component of $\chi$ where $\mathfrak{q}$ is a  prime ideal of $\mathcal{O}_K$.
Let $\mathfrak{a}$ be a fractional ideal prime to $\mathfrak{c}$ and  $a$ be a idele with  $a\widehat{\mathcal{O}}_K \cap K = \mathfrak{a}$
such that $a_{\mathfrak{q}} = 1$ for all $\mathfrak{q}\mid \mathfrak{c}$.
If $\chi$ has conductor $\mathfrak{c}$ then we write $\chi(\mathfrak{a})$ for $\chi(a)$.

The $p$-adic avatar $\widehat{\chi} : K^\times \backslash \widehat{K}^\times  \to \mathbb{C}^\times_p$ of $\chi$ having infinity type $(m,n)$
is given as
$$\widehat{\chi}(z) := \iota_p \circ \iota^{-1}_{\infty} \left( \chi(z) \right) \cdot z^m_{\mathfrak{p}} \cdot z^n_{\overline{\mathfrak{p}}}$$
for $z \in \widehat{K}^\times$.

Let $f\in S^{\mathrm{new}}_{2r}(\Gamma_0(N))$  be an elliptic newform with level $N_f$ dividing $N$ with $q$-expansion
$$f(q)=\sum_{n>0}\mathbf{a}_n(f) q^n.$$
Suppose $F$ is a finite extension of $\Q_p$ containing the field generated by Fourier coefficients $\{\mathbf{a}_n(f)\}_n$ over $\Q$. 
One can show that  there is a unique geometric modular form $f^\flat$ of weight $2r$ and level $\Gamma_0(Np^2)$ defined over $\calO_F$ such that:
\begin{itemize}\item $f^\flat(A_x,\eta_x,2\pi i dw)=\mathbf{f}^\flat(x)$ for $x\in \mathbb{H}\times\GL_2(\hat{\Q})$ and a complex function $\mathbf{f}^\flat: \mathbb{H} \times \GL_2(\hat{\Q}) \rightarrow \C$ as defined in \cite[eq. (3.5)]{castella-hsieh-heegner-cycles}, and 
\item $f^\flat({\mathrm{ Tate}}(q),\eta_{\mathrm{can}},\omega_{\mathrm{can}})=\sum_{p\nmid n}\mathbf{a}_n(f) q^n.$
\end{itemize}
 The associated $p$-adic modular form $\widehat{ f}^\flat\in V_p(N,\calO_F)$ is of weight $2r$.
 Let $c=c_0 p^n$ with $p\nmid c_0$ and $n\geq 0$. Define
\[\Pic\calO_c:=K^\times\backslash\widehat{K}^\times/\widehat{\calO}_c^\times.\]
Let $a\in \widehat{K}^{(c\gp\mathfrak{N})\times}$ and $\fraka=a\widehat{K}\cap\calO_c$, we will write $[a]=[\fraka]$ to denote its class in $\Pic\calO_c$. Let $\chi:K^\times\backslash \mathbb{A}_K^\times/\widehat{\calO}_c^\times\to\C^\times$ be an anticyclotomic Hecke character, and define
\[A(\chi)=\{\textrm{primes $q\mid D_K$ such that $q\mid N_f$ and the character $\chi_q$ is unramified.}\}.\]
In \cite{castella-hsieh-heegner-cycles}, the following hypotheses are assumed:
\begin{itemize}
    \item[\textbf{(Heeg')}]  $N_f^-$ is a square-free product of primes ramified in $K$, and
    \item[\textbf{(ST)}] $\mathbf{a}_q(f)\chi(\mathfrak{q})=-1$ for every $q\in A(\chi)\quad (q\calO_K=\mathfrak{q}^2)$.
\end{itemize}
Since we assume $(D_K, N)= 1$ in \S\ref{subsec:setup}, these hyotheses are automatic.

\subsection{BDP $p$-adic $L$-functions}\label{sec:2.4BDp}
Let $K[p^\infty] = \cup_{n \geq 1}K[p^n]$ be the ring class field of conductor $p^\infty$
and $\widetilde{\Gamma} = \mathrm{Gal}(K[p^\infty] / K)$.
Let $\Gamma$ be the maximal free quotient of $\widetilde{\Gamma}$.
Let $\mathcal{C}( \widetilde{\Gamma} , \mathcal{O}_{\mathbb{C}_p} )$
be the the space of continuous $\mathcal{O}_{\mathbb{C}_p}$-valued functions on $\widetilde{\Gamma}$
and let 
$\mathfrak{X}_{p^\infty} \subseteq \mathcal{C}( \widetilde{\Gamma} , \mathcal{O}_{\mathbb{C}_p} )$ be the
the set of locally algebraic $p$-adic characters $\rho : \widetilde{\Gamma} \to \mathcal{O}^{\times}_{\mathbb{C}_p}$.

Let $\mathrm{rec}_{\mathfrak{p}}$ be the local reciprocity map given by 
$$\mathrm{rec}_{\mathfrak{p}} : \mathbb{Q}^\times_p = K^\times_{\mathfrak{p}} \to \mathrm{Gal}(K^{\mathrm{ab}}/K) \to \widetilde{\Gamma}.$$
 For $\rho\in \mathfrak{X}_{p^\infty}$, define the map $\rho_{\mathfrak{p}}: \Q_p^\times \rightarrow \mathbb{C}_p^\times$ as $\rho_{\mathfrak{p}}(\beta) = \rho(\mathrm{rec}_{\gp}(\beta))$.  For $\rho \in \calC(\tilde{\Gamma}, \calO_{\mathbb{C}_p}),$ define $\rho|[\ga]:\Z_p^\times \rightarrow \calO_{\mathbb{C}_p}$ as $$\rho|[\ga](x)= \rho(\mathrm{rec}_{\mathfrak{p}}(x)\sigma_{\mathfrak{a}}^{-1})=\rho(\mathrm{rec}_\gp(x)\mathrm{rec}_K(a)).$$ 

For every $a\in\widehat{K}^{(c_0 \mathfrak{N} p)\times}$ consider the pair $(\mathfrak{a},a)$ where $\mathfrak{a}$ is the associated fractional ideal in $\calO_{c_0}$. Recall that to a pair $(\mathfrak{a},a)$ we have an associated CM point
 $(A_\fraka,\eta_\fraka)$.
Let $t_\fraka$ be the canonical Serre--Tate coordinate of $\widehat{f}^\flat$ around $\mathbf{x}_\fraka=[(A_\fraka,\eta_\fraka)]\otimes_\calW\bar{\mathbb{F}}_p$, and define
$$\widehat{f}^\flat_{\fraka}(t_\fraka):=\widehat{ f}^\flat(t_\fraka^{\mathrm{N}(\fraka)^{-1}\sqrt{-D_K}^{-1}})\in\calW \llbracket t-1 \rrbracket \quad( \mathrm{where}\text{ }\mathrm{N}(\fraka)=|a|_{\mathbb{A}_K}^{-1}{c_0}^{-1}).$$
We recall the analytic definition of the BDP $p$-adic $L$-function as in \cite[Definition 3.7]{castella-hsieh-heegner-cycles}. 
\begin{definition}\label{D:padicL}
Suppose $\psi$ is an anticyclotomic Hecke character of infinity type $(r,-r)$
 and let $c_0\calO_K$ be the prime-to-$p$ part of the conductor of $\psi$. 
 Define the $p$-adic measure $\mathscr{L}_{\gp,\psi}(f)$ on $\widetilde{\Gamma}$ by
\[\mathscr{L}_{\gp,\psi}(f)(\rho)=\sum_{[\fraka]\in \Pic\calO_{c_0}}\psi(\fraka)\mathrm{N}(\fraka)^{-r}\cdot \left(\widehat{ f}^\flat_{\fraka}\otimes\psi_\gp\rho|[\fraka]\right)(A_\fraka,\eta_\fraka)\]
where $\rho \in \mathfrak{X}_{p^\infty}$.
We will view $\mathscr{L}_{\gp,\psi}(f)$ as an element in the semi-local ring $\calW \llbracket \widetilde{\Gamma} \rrbracket$.
\end{definition}

\subsection{The finite layer analogue of BDP $p$-adic $L$-functions} \label{sec:2.5BDP}
By using the factorization of Gauss sums and the interpolation formula for BDP $p$-adic $L$-functions, we construct an analogue of Bertolini--Darmon's theta elements for the BDP setting, the finite layer version of BDP $p$-adic $L$-functions.
\subsubsection{Factorizing Gauss sums}
Assume that $\hat{\phi}\in \mathfrak{X}_{p^\infty}$ is the $p$-adic avatar of a Hecke character $\phi$ with infinity type $(0,0)$, so $\widehat{\phi} = \phi$ and it has finite order. Assume $c_0 = 1$.
Let $\chi = \psi \cdot \phi$ be an anticyclotomic character of infinity type $(r, -r)$. From Definition \ref{D:padicL} and Proposition \ref{prop:tensor} we obtain the explicit formula
\begin{align*}
\mathscr{L}_{\mathfrak{p}, \psi}(f)(\phi) & = \sum_{[a] \in \mathrm{Pic}(\mathcal{O}_{K[1]})} \left( \widehat{f}^{\flat}_{\mathfrak{a}} \otimes (\chi \circ \mathrm{rec}_{\mathfrak{p}}) \right) (x_{\mathfrak{a}}) \cdot \left(\chi \cdot \vert - \vert^{r}_{\mathbb{A}_K} \right)(a) \\
& = p^{-n} \cdot \mathfrak{g}(\chi \circ \mathrm{rec}_{\mathfrak{p}}) \\
& \ \ \ \ \times \sum_{[a] \in \mathrm{Pic}(\mathcal{O}_{K[1]})} \left(\overline{\chi} \cdot \vert - \vert^{r}_{\mathbb{A}_K} \right)(a) \cdot \sum_{u \in (\mathbb{Z}_p/p^n\mathbb{Z})^\times} \widehat{f}^{\flat}(x_{\mathfrak{a}} * \mathbf{n}(up^{-n})) \cdot \left( \chi \circ \mathrm{rec}_{\mathfrak{p}} \right) (u^{-1})
\end{align*}
where $\overline{\chi}$ is the character on $\mathrm{Pic}(\mathcal{O}_{K[1]})$ induced from $\chi$.
We compute
\begin{align*}
& p^{-n} \cdot \mathfrak{g}(\chi \circ \mathrm{rec}_{\mathfrak{p}}) \cdot \sum_{[a] \in \mathrm{Pic}(\mathcal{O}_{K[1]})} \left(\overline{\chi} \cdot \vert - \vert^{r}_{\mathbb{A}_K} \right)(a) \cdot \sum_{u \in (\mathbb{Z}_p/p^n\mathbb{Z})^\times} \widehat{f}^{\flat}(x_{\mathfrak{a}} * \mathbf{n}(up^{-n})) \cdot \left( \chi \circ \mathrm{rec}_{\mathfrak{p}} \right) (u^{-1}) \\
& = 
p^{-n} \cdot \left( \sum_{u' \in (\mathbb{Z}/p^n\mathbb{Z})^\times} \left(\chi \circ \mathrm{rec}_{\mathfrak{p}} \right) (u') \cdot \zeta^{u'}_{p^n} \right) \\
& \ \ \ \ \times \sum_{[a] \in \mathrm{Pic}(\mathcal{O}_{K[1]})} \left(\overline{\chi} \cdot \vert - \vert^{r}_{\mathbb{A}_K} \right)(a) \cdot \sum_{u \in (\mathbb{Z}_p/p^n\mathbb{Z})^\times} \widehat{f}^{\flat}(x_{\mathfrak{a}} * \mathbf{n}(up^{-n})) \cdot \left( \chi \circ \mathrm{rec}_{\mathfrak{p}} \right) (u^{-1})    \\
& = 
p^{-n} \cdot \sum_{[a] \in \mathrm{Pic}(\mathcal{O}_{K[1]})} \left(\overline{\chi} \cdot \vert - \vert^{r}_{\mathbb{A}_K} \right)(a) \\
& \ \ \ \
\times  \sum_{u' \in (\mathbb{Z}/p^n\mathbb{Z})^\times} 
\sum_{u \in (\mathbb{Z}_p/p^n\mathbb{Z})^\times} 
\zeta^{u'}_{p^n} \cdot \widehat{f}^{\flat}(x_{\mathfrak{a}} * \mathbf{n}(up^{-n})) \cdot
\left( \chi \circ \mathrm{rec}_{\mathfrak{p}} \right) (u^{-1} \cdot u') \\
& = 
p^{-n} \cdot \sum_{[a] \in \mathrm{Pic}(\mathcal{O}_{K[1]})} \left(\overline{\chi} \cdot \vert - \vert^{r}_{\mathbb{A}_K} \right)(a) \\
& \ \ \ \ \times  \sum_{c \in (\mathbb{Z}/p^n\mathbb{Z})^\times} 
\sum_{u \in (\mathbb{Z}_p/p^n\mathbb{Z})^\times} 
\zeta^{uc}_{p^n} \cdot \widehat{f}^{\flat}(x_{\mathfrak{a}} * \mathbf{n}(up^{-n})) \cdot
\left( \chi \circ \mathrm{rec}_{\mathfrak{p}} \right) (c) 
\end{align*}
where we put $u' = u\cdot c \in (\mathbb{Z}/p^n\mathbb{Z})^\times$.
\subsubsection{BDP theta elements}\label{def_theta}
From the above computation, it is natural to define the theta element $\theta_n(f)$ by
$$\theta_n(f) := p^{-n}  \cdot   \sum_{c \in (\mathbb{Z}/p^n\mathbb{Z})^\times} \left( \sum_{u \in (\mathbb{Z}_p/p^n\mathbb{Z})^\times} \zeta^{uc}_{p^n} \cdot \sum_{[a] \in \mathrm{Pic}(\mathcal{O}_{K[1]})} \vert a \vert^{r}_{\mathbb{A}_K}   \cdot \widehat{f}^{\flat}(x_{\mathfrak{a}} * \mathbf{n}(up^{-n})) \right) \cdot \widetilde{a} \cdot  c     $$
where $\widetilde{a}$ runs over a lifting of $\mathrm{Pic}(\mathcal{O}_{K[1]})$ to $\widetilde{\Gamma}_n$.
Due to the choice of lifting, $\theta_n(f) $ is well-defined only up to multiplication by an element of $\widetilde{\Gamma}_n$.
See \cite[\S2.3.5]{kim-summary} for a similar phenomenon for Bertolini--Darmon's theta elements.
Since $\phi$ is a finite order character and $\mathscr{L}_{\mathfrak{p}, \psi}(f)$ is integral,
$\theta_n(f)$ also lies in $\mathcal{W}[\widetilde{\Gamma}_n] \subseteq \mathbb{C}_p[\widetilde{\Gamma}_n]$.
To sum up, we have the slightly twisted comparison between $\mathscr{L}_{\mathfrak{p}, \psi}(f)$ and  $\theta_n(f)$ as
$$\mathscr{L}_{\mathfrak{p}, \psi}(f)(\phi) = \phi \left( \mathscr{L}_{\mathfrak{p}, \psi}(f) \right)
= (\psi \cdot \phi) ( \theta_n(f) ) = \chi( \theta_n(f) ) .$$
Let $\Lambda^{\psi^{-1}}$ be the $\psi^{-1}$-twisted Iwasawa algebra defined by isomorphism
$\Lambda^{\psi^{-1}} \simeq \Lambda$ sending $\sigma$ to $\psi(\sigma) \cdot \sigma$ for a group-like element $\sigma$.
The same notational rule applies to finite layer group rings.
The above twisted comparison shows that if $\mathscr{L}_{\mathfrak{p}, \psi}(f)$ lies in $\Lambda^{\psi^{-1}}$ then $\theta_n(f)$ is the image of $\mathscr{L}_{\mathfrak{p}, \psi}(f)$ under $\Lambda^{\psi^{-1}} \simeq \Lambda \to \Lambda_n$.

Denote by $\iota$ the involution on $\mathcal{W}[\widetilde{\Gamma}_n]$ defined by sending a group-like element to its inverse. Then $\theta_n(f) \cdot \iota( \theta_n(f) )$ is independent of the choice of lifting.

It is not difficult to see that $\mathscr{L}_{\mathfrak{p}, \psi}(f) \cdot \mathscr{L}_{\mathfrak{p}, \psi}(f)$ and $\mathscr{L}_{\mathfrak{p}, \psi}(f) \cdot \iota (\mathscr{L}_{\mathfrak{p}, \psi}(f))$ differ only up to multiplication by an element of $\widetilde{\Gamma}$ by the functional equation for generalized Heegner cycles \cite[Lem. 4.6, Prop. 7.4]{castella-hsieh-heegner-cycles} and the explicit reciprocity law for generalized Heegner cycles \cite[Thm. 5.7]{castella-hsieh-heegner-cycles}.

\section{BDP Selmer groups and the strong main conjectures}\label{sec:BDPselmer}
From this section onward, we will restrict to the case of elliptic curves over $\Q$. Let $f\in S_2(\Gamma_0(N))$ be the newform attached to the elliptic curve $E$ over $\Q$. Let $\psi$ be an anticyclotomic Hecke character of infinity type $(1,-1)$. As in \cite[p. 540, line -2]{CGLS2022}, for every character $\xi$ of $\Gamma$, define $\mathcal{L}_E\in \Lambda^\mathrm{ur}$ as $$\mathcal{L}_E(\xi)=\mathscr{L}_{\gp,\psi}(\psi^{-1}\xi)^2.$$ 
This identification reflects the $\psi^{-1}$-twist in $\S$\ref{def_theta}.
\subsection{BDP Selmer groups}\label{sec:BDP Selmer}

Let $K_\infty$ be the anticyclotomic $\Z_p$-extension of $K$ with subextension $K_n$ of degree $p^n$ over $K$ and let $\Lambda$ be the Iwasawa algebra $\Z_p \llbracket \Gal(K_\infty/K) \rrbracket$. Set $\Lambda^\mathrm{ur}:=\Lambda \hat{\otimes}_{\Z_p} \Z_p^{\mathrm{ur}}$ where $\Z_p^{\mathrm{ur}}$ is the completion of the ring of integers of the maximal unramified extension of $\Q_p$. Assume that the prime $p$ splits in $\mathcal{O}_K$ as $p\mathcal{O}_K=\gp \cdot \gp^c$ and since $(p, h_K)=1$, the primes $\gp$ and $\gp^c$ are totally ramified in $K_\infty$.

Let $T$ be the $p$-adic Tate module of $E$ and set $\bA:=T\hat{\otimes}_{\Z_p}\Hom_{\Z_p}(\Lambda, \Q_p/\Z_p)$ equipped with the natural  $G_K$-action, which means that $G_K$ acts on $\Lambda$ by the tautological character. One defines the \textbf{BDP-Selmer group} as
\begin{equation}\label{def1}
    \Sel^{\emptyset,0}(K,\bA):=\ker\Big(H^1(K,\bA) \rightarrow \prod_{\gq \in \{\gp, \gp^c\}}\frac{H^1(K_\gq,\bA)}{H^1_{\mathcal{L}_\gq}(K_\gq,\bA)}\times \prod_{\gq \nmid p, \gq \mid N}H^1(K_\gq,\bA)\Big),
\end{equation}

where $H^1_{\mathcal{L}_\gp}(K_\gq,\bA)=H^1(K_\gp, \bA)$ and $H^1_{\mathcal{L}_{\gp^c}}(K_\gq,\bA)=0$.

Let $V=T[\frac{1}{p}]$ and $A=V/T \cong E[p^\infty]$. By Shapiro's lemma, one obtains that $\Sel^{\emptyset,0}(K,\bA) \cong \Sel^{\emptyset,0}(K_\infty,A)$ where the definition of the later Selmer group is the same as definition \eqref{def1} with $(K,\bf{A})$ replaced by $(K_\infty,A)$. Let $X^{\emptyset,0}$ be the Pontryagin dual of $\Sel^{\emptyset,0}(K_\infty,A)$. 

 \begin{remark}
In the definition of $\Sel^{\emptyset,0}(K,\bA)$ above, we have taken the trivial local conditions for primes away from $p$.
However, the standard BDP Selmer group (see example \cite{KobOta}) is defined using the unramified local conditions at primes away from $p$. The possible difference between these two local conditions only comes from primes dividing $N^-$ (see \cite[p. 1363, line 11]{Pollack_Weston_2011}). By \cite[Definition 3.3, Lemma 3.4]{Pollack_Weston_2011}, the hypothesis (\ref{Tam}) for primes dividing $N^-$ ensures that this difference is trivial and hence our definition of the BDP Selmer group coincides with the standard one.
\end{remark}

 In the following, we record the current state of art on the Greenberg--Iwasawa main conjecture. Let $E/\Q$ be an elliptic curve of conductor $N$. Using the imaginary quadratic field $K$, we decompose
$$N = N^+ \cdot N^-$$
where a prime divisor of $N^-$ is inert in $K/\mathbb{Q}$ and a prime divisor of $N^+$ splits  or is ramified  in $K/\mathbb{Q}$.

\begin{remark}
    This choice of $N^+$ and $N^-$ does not match with \cite[p. 575, line -1]{castella-hsieh-heegner-cycles} but matches with \cite{castella-wan-perrin-riou-ss} and hence we will assume henceforth that $(N, pD_K)=1$ so that it agrees with \cite[p. 575]{castella-hsieh-heegner-cycles}. 
\end{remark}
 \subsubsection{Ordinary case}\label{sec:ord}
 We make the following assumptions in the ordinary case.
 \begin{enumerate}
     \item $E[p]$ is irreducible as a $G_\Q$-module,
     \item the discriminant $D_K<0$ is odd,  and $D_K \neq -3$.
     \item every prime $\ell \mid N$ splits in $K$,
     \item $p$ splits in $K$ and $p >3$,
     \item $\bar{\rho}_E: G_\Q \rightarrow \mathrm{Aut}_{\mathbb{F}_p}(E[p])$ is surjective.
 \end{enumerate}
The following theorem is \cite[Theorem 1.2.4]{burungale-castella-skinner-gl2}.
\begin{theorem}\label{thm:main-conj-ord}
    Let $E/\Q$ has good ordinary reduction at $p$. Under the hypotheses (1)-(5) above, $X^{\emptyset,0}$ is $\Lambda$-torsion and $$\mathrm{Char}_{\Lambda}(X^{\emptyset,0})\Lambda^{\mathrm{ur}}=(\mathcal{L}_E)$$ as ideals in $\Lambda^{\mathrm{ur}}.$
\end{theorem}

\begin{remark}
    If $E[p]$ is reducible, then  we still have $\mathrm{Char}_{\Lambda}(X^{\emptyset,0})\Lambda^{\mathrm{ur}}=(\mathcal{L}_E)$ under certain hypotheses (cf. \cite[Theorem C]{CGLS2022}). 
\end{remark}

\subsubsection{The supersingular case}\label{sec:supersingular}
When $E/\Q$ has supersingular reduction at $p$, we assume the following conditions:
\begin{enumerate}
\item $N^-$ is the square-free product of an even number of primes.
\item $p$ splits in $K/\mathbb{Q}$ and write $(p) = \mathfrak{p} \overline{\mathfrak{p}}$ in $\mathcal{O}_K$
and $\mathfrak{p}$ is compatible with the choice of $\overline{\mathbb{Q}}_p$.
\item $N$ is square-free.
\item There exists a prime $\ell$ dividing $N$ such that $\ell$ is non-split in $K/\mathbb{Q}$.
\item If $N$ is odd, then 2 splits in $K/\mathbb{Q}$.
\item $\overline{\rho}_E$ is ramified at every prime dividing $N^-$.
\end{enumerate}
\begin{theorem}\label{thm:main-conj-ss}
   Let $E$ has good supersingular reduction at $p$ with $a_p=p+1-\#\tilde{E}(\mathbb{F}_p)=0.$ Under the hypotheses (1)-(6) above, $X^{\emptyset,0}$ is $\Lambda$-torsion and $$\mathrm{Char}_{\Lambda}(X^{\emptyset,0})\Lambda^{\mathrm{ur}}=(\mathcal{L}_E)$$ as ideals in $\Lambda^{\mathrm{ur}}.$
\end{theorem}
\begin{proof}
   This follows by combining Theorem 5.3, Theorem 6.8 and Theorem A.5 of \cite{castella-wan-perrin-riou-ss}.
\end{proof}

\section{The proof of the strong main conjecture}
\subsection{More ingredients}\label{quick}
 
For each prime number $v$ of $K$, let $c_v(E):=\big|E(K_v)/E_0(K_v)\big|$ be the Tamagawa number of $E$ at $v$. The  assumption we impose is
\begin{equation}\label{Tam}
\tag{Tam} \text{$p\nmid c_v(E)$ for all primes $v\,|\,N$.}
\end{equation}
We also assume the following condition which is also an assumption in \cite[p.154]{LMX}.
\begin{equation}\label{H^0}
\tag{$\mathrm{H^0}$} \text{$H^0(K_v,A)=0$ for all primes $v\,|\,p$.}
\end{equation}
We obtain the following control theorem.
\begin{theorem} \label{thm:control-theorem}
    Under \eqref{Tam} and \eqref{H^0}, one has that the natural homomorphism 
    $$\Sel^{\emptyset,0}(K_n,A) \rightarrow \Sel^{\emptyset,0}(K_\infty,A)^{\Gamma_n}$$ is an isomorphism.
\end{theorem}
\begin{proof}
    Suppose $E$ has supersingular reduction at $p$. Then we note that $A(K_\infty)$ and $A(K_{\infty,\gp^c})$ are both trivial from \cite[proof of Theorem 4.1]{LLMAdvances} and hence $A(K_\infty)$ and $A(K_{\infty,\gp^c}){}$ are also trivial. If $E$ has good ordinary reduction at $p$, then hypothesis \eqref{H^0} ensures that these groups are trivial. Therefore, the proof follows with the arguments exactly as in \textit{(loc.cit)}. We note that the failure of the surjectivity comes from the prime-to-$p$ local conditions and and this coincides with the ordinary case. Hypothesis \textbf{(Tam)} ensures that this failure vanishes (cf.  \cite[Proof of Theorem 3.9]{kim-kurihara}).
\end{proof}

\begin{theorem} \label{thm:non-existence-finite}
    The $\Lambda$-submodule $\Sel^{\emptyset,0}(K_\infty,A)^\vee$ is $\Lambda$-torsion and does not contain any non-trivial finite submodule.
\end{theorem}
\begin{proof}
The torsion-ness result follows from 
\cite[Theorem 1.6]{KobOta} which deals with both the ordinary and the non-ordinary cases simultaneously. Here we note that the Bloch-Kato Selmer groups in \textit{(loc.cit)} in in fact equal to our BDP Selmer group; thanks to \cite[p. 548, l. 1]{KobOta}, \cite[Remark 2.6]{KobOta} and \cite[p. 551, l. 14]{KobOta}. (See also the proof of Theorem 6.2 of \cite{castella-hsieh-heegner-cycles} which uses  \cite[Remark 2.6]{KobOta})

    The proof of the non-existence of finite submodule follows the same argument as laid by Greenberg in \cite[Prop. 4.14]{greenberg-lnm} (see also \cite[Lemma 3.4]{LMX} for the argument in the BDP-case which essentially follows Greenberg's original proof).
\end{proof}

\subsection{Putting it all together}\label{final}

By the main conjecture for BDP Selmer groups (Theorems \ref{thm:main-conj-ord} and \ref{thm:main-conj-ss}), we have
$$\mathrm{Char}_{\Lambda}(X^{\emptyset,0})\Lambda^{\mathrm{ur}}=(\mathcal{L}_E)$$
in $\Lambda^{\mathrm{ur}}$.
By the non-existence of non-trivial finite $\Lambda$-submodule of $\Sel^{\emptyset,0}(K_\infty,A)^\vee$ (Theorem \ref{thm:non-existence-finite}),
we have
$$\mathcal{L}_E = \mathrm{Fitt}_{\Lambda}(\Sel^{\emptyset,0}(K_\infty,A)^\vee).$$
Since Fitting ideals behave well under base change, we have
$$\mathcal{L}_E \pmod{\omega_n} = \mathrm{Fitt}_{\Lambda/\omega_n}(\Sel^{\emptyset,0}(K_\infty,A)[\omega_n]^\vee)$$
in $\Lambda_n = \Lambda/\omega_n$. 
By control theorem (Theorem \ref{thm:control-theorem}), we have
$$\mathcal{L}_E \pmod{\omega_n} = \mathrm{Fitt}_{\Lambda_n}(\Sel^{\emptyset,0}(K_n,A)^\vee).$$
Since we computed
$$\mathcal{L}_E \pmod{\omega_n} = \theta_n(f) \cdot \iota( \theta_n(f))$$
in $\S$\ref{sec:2.5BDP}, we have the conclusion
$$\theta_n(f) \cdot \iota( \theta_n(f)) = \mathrm{Fitt}_{\Lambda_n}(\Sel^{\emptyset,0}(K_n,A)^\vee).$$

\bibliographystyle{alpha}
\bibliography{library}
\end{document}